\documentclass[12pt]{article}
\usepackage{amsmath,amssymb,amsthm}
\numberwithin{equation}{section}
\usepackage{hyperref}
\usepackage{units}
\usepackage{color}
\usepackage[T1]{fontenc}
\usepackage[utf8]{inputenc}
\usepackage{authblk}
\usepackage{bm}
\usepackage{graphicx}

\usepackage{datetime}

\usepackage[toc,page]{appendix}

\newcommand{\Z}{{\mathbb Z}}

\newtheorem{thm}{Theorem}

\newtheorem{lemma}{Lemma}

\title{Squares of Tribonacci numbers}

\author[]{Kunle Adegoke \thanks{adegoke00@gmail.com}}

\affil{Department of Physics and Engineering Physics, \mbox{Obafemi Awolowo University}, 220005 Ile-Ife, Nigeria}

\begin{document}

\date{}

\maketitle

\begin{abstract}
\noindent We prove some identities for the squares of generalized Tribonacci numbers. Various summation identities involving these numbers are derived.
\end{abstract}
\tableofcontents
\section{Introduction}
For $r\ge 3$, the generalized Tribonacci numbers are defined by
\begin{equation}\label{eq.ffnh4ol}
\mathcal{T}_r=\mathcal{T}_{r-1}+\mathcal{T}_{r-2}+\mathcal{T}_{r-3}\,,
\end{equation}
with arbitrary initial values $\mathcal{T}_0$, $\mathcal{T}_1$ and $\mathcal{T}_2$.

\bigskip

By substituting $\mathcal{T}_{r-2}=\mathcal{T}_{r-1}-\mathcal{T}_{r-3}-\mathcal{T}_{r-4}$ into the recurrence relation~\eqref{eq.ffnh4ol}, a useful alternative recurrence relation is obtained for $r\ge 4$: 
\begin{equation}\label{eq.o6ns6qw}
\mathcal{T}_r=2\mathcal{T}_{r-1}-\mathcal{T}_{r-4}\,,
\end{equation}
with arbitrary initial values $\mathcal{T}_0$, $\mathcal{T}_1$, $\mathcal{T}_2$ and $\mathcal{T}_3$.

\bigskip

Extension of the definition of $\mathcal{T}_r$ to negative subscripts is provided by writing the recurrence relation~\eqref{eq.o6ns6qw} as
\begin{equation}
\mathcal{T}_{-r}=2\mathcal{T}_{-r+3}-\mathcal{T}_{-r+4}\,.
\end{equation}
When $\mathcal{T}_0=0$, $\mathcal{T}_1=1$ and $\mathcal{T}_2=1$, we have the Tribonacci sequence, denoted $\{T_r\}$, $r\in\Z$.

\bigskip

Our purpose in writing this note is to establish the following identities:
\[
\mathcal{T}_r^2  - 2\mathcal{T}{}_{r - 1}^2  - 3\mathcal{T}_{r - 2}^2  - 6\mathcal{T}_{r - 3}^2  + \mathcal{T}_{r - 4}^2  + \mathcal{T}_{r - 6}^2  = 0\,,
\]
\[
\mathcal{T}_{r + 2}^2  - 4\mathcal{T}_{r + 1}^2  + \mathcal{T}_r^2  + 14\mathcal{T}_{r - 2}^2  - 4\mathcal{T}_{r - 3}^2  - 2\mathcal{T}_{r - 4}^2  - 8\mathcal{T}_{r - 5}^2  + \mathcal{T}_{r - 6}^2  + \mathcal{T}_{r - 8}^2=0\,,
\]
\[
\mathcal{T}_{r + 3}^2  - 3\mathcal{T}_{r + 2}^2  - 4\mathcal{T}_r^2  + 2\mathcal{T}_{r - 1}^2  - 10\mathcal{T}_{r - 2}^2  - 4\mathcal{T}_{r - 3}^2  + \mathcal{T}_{r - 5}^2  + \mathcal{T}_{r - 6}^2  = 0\,,
\]
\[
\mathcal{T}_{r + 1}^2  - 4\mathcal{T}_r^2  + 2\mathcal{T}_{r - 2}^2  + 16\mathcal{T}_{r - 3}^2  + 4\mathcal{T}_{r - 4}^2  - 2\mathcal{T}_{r - 6}^2  - \mathcal{T}_{r - 7}^2  = 0
\]
and
\[
\mathcal{T}_{r + 2}^2  - 2\mathcal{T}_{r + 1}^2  - 2\mathcal{T}_r^2  - 8\mathcal{T}_{r - 1}^2  - 2\mathcal{T}_{r - 2}^2  - 6\mathcal{T}_{r - 3}^2  + 2\mathcal{T}_{r - 4}^2  + \mathcal{T}_{r - 6}^2  = 0\,.
\]
We will also develop the associated summation identities. Specifically, we shall evaluate
\[
\begin{split}
&\sum_{j = 0}^k {x^j \mathcal{T}_j^2 },\quad \sum_{j = 0}^k {T_{2j}^2 },\quad \sum_{j = 1}^k {T_{2j - 1}^2 },\quad \sum_{j = 0}^k {( - 1)^j T_{2j}^2 },\quad \sum_{j = 1}^k {( - 1)^{j - 1} T_{2j - 1}^2 }\,,\\
& \sum_{j = 0}^k {T_{4j}^2 },\quad\sum_{j = 1}^k {T_{4j - 3}^2 },\quad\sum_{j = 0}^k {T_{4j - 1}^2 },\quad\sum_{j = 0}^k {j\mathcal{T}_j^2 },\quad\sum_{j = 0}^k {j^2\mathcal{T}_j^2 }\text{ and } \sum_{j = 0}^\infty {x^j \mathcal{T}_j^2 }\,.
\end{split}
\]

\bigskip

Presently we derive some identities that we shall need.

\bigskip

Rearranging the identity~\eqref{eq.o6ns6qw} and squaring, we have
\begin{equation}\label{eq.pbljs54}
4\mathcal{T}_{r - 1} \mathcal{T}_r  = 4\mathcal{T}_{r - 1}^2  - \mathcal{T}_{r - 4}^2  + \mathcal{T}_r^2\,,
\end{equation}
\begin{equation}\label{eq.a41lrc2}
4\mathcal{T}_{r - 1} \mathcal{T}_{r - 4}  = 4\mathcal{T}_{r - 1}^2  + \mathcal{T}_{r - 4}^2  - \mathcal{T}_r^2
\end{equation}
and
\begin{equation}\label{eq.qjbtgy8}
2\mathcal{T}_r \mathcal{T}_{r - 4}  = 4\mathcal{T}_{r - 1}^2  - \mathcal{T}_{r - 4}^2  - \mathcal{T}_r^2\,.
\end{equation}
Rearranging identity~\eqref{eq.o6ns6qw} and multiplying through by $4T_{r-3}$ to obtain
\begin{equation}
8\mathcal{T}_{r - 1} \mathcal{T}_{r - 3}  = 4\mathcal{T}_r \mathcal{T}_{r - 3}  + 4\mathcal{T}_{r - 4} \mathcal{T}_{r - 3}\,,
\end{equation}
and using identities~\eqref{eq.pbljs54} and~\eqref{eq.a41lrc2} to resolve the right hand side gives
\begin{equation}\label{eq.y9lf0gg}
8\mathcal{T}_{r - 1} \mathcal{T}_{r - 3}  = 4\mathcal{T}_r^2  + 2\mathcal{T}_{r - 3}^2  - \mathcal{T}_{r + 1}^2  + 4\mathcal{T}_{r - 4}^2  - \mathcal{T}_{r - 7}^2\,.
\end{equation}

\bigskip

We also require the following results from a previous paper~\cite{adegoke18c}.
\begin{lemma}[Partial sum of an $n^{th}$ order sequence]\label{lemma.qm8k37h}
Let $\{X_j\}$ be any arbitrary sequence, where $X_j$, $j\in\Z$, satisfies a $n^{th}$~order recurrence relation $X_j=f_1X_{j-c_1}+f_2X_{j-c_2}+\cdots+f_nX_{j-c_n}=\sum_{m=1}^n f_mX_{j-c_m}$, where $f_1$, $f_2$, $\ldots$, $f_n$ are arbitrary non-vanishing complex functions, not dependent on $j$, and $c_1$, $c_2$, $\ldots$, $c_n$ are fixed integers. Then, the following summation identity holds for arbitrary $x$ and non-negative integer $k$ :
\[
\sum_{j = 0}^k {x^j X_j }  = \frac{{\sum_{m = 1}^n {\left\{ {x^{c_m } f_m \left( {\sum_{j = 1}^{c_m } {x^{ - j} X_{ - j} }  - \sum_{j = k - c_m  + 1}^k {x^j X_j } } \right)} \right\}} }}{{1 - \sum_{m = 1}^n {x^{c_m } f_m } }}\,.
\]

\end{lemma}
\begin{lemma}[Generating function]\label{lemma.v1j9biq}
Under the conditions of Lemma~\ref{lemma.qm8k37h}, if additionally $x^kX_k$ vanishes in the limit as $k$ approaches infinity, then
\[
S_\infty  (x) = \sum_{j = 0}^\infty  {x^j X_j }  = \frac{{\sum_{m = 1}^n {\left( {x^{c_m } f_m \sum_{j = 1}^{c_m } {x^{ - j} X_{ - j} } } \right)} }}{{1 - \sum_{m = 1}^n {x^{c_m } f_m } }}\,,
\]
so that $S_\infty(x)$ is a generating function for the sequence $\{X_j\}$.
\end{lemma}
We note that a special case of Lemma~\ref{lemma.qm8k37h} was proved in~\cite{zeitlin64}.

\section{Squares of Tribonacci numbers}
\begin{thm}\label{thm.eo07gjo}
The following identities hold for any integer $r$:
\begin{equation}\label{eq.jesyg5e}
\mathcal{T}_r^2  - 2\mathcal{T}{}_{r - 1}^2  - 3\mathcal{T}_{r - 2}^2  - 6\mathcal{T}_{r - 3}^2  + \mathcal{T}_{r - 4}^2  + \mathcal{T}_{r - 6}^2  = 0\,,
\end{equation}
\begin{equation}\label{eq.ltwcc00}
\mathcal{T}_{r + 2}^2  - 4\mathcal{T}_{r + 1}^2  + \mathcal{T}_r^2  + 14\mathcal{T}_{r - 2}^2  - 4\mathcal{T}_{r - 3}^2  - 2\mathcal{T}_{r - 4}^2  - 8\mathcal{T}_{r - 5}^2  + \mathcal{T}_{r - 6}^2  + \mathcal{T}_{r - 8}^2=0\,,
\end{equation}
\begin{equation}\label{eq.pq5tubq}
\mathcal{T}_{r + 3}^2  - 3\mathcal{T}_{r + 2}^2  - 4\mathcal{T}_r^2  + 2\mathcal{T}_{r - 1}^2  - 10\mathcal{T}_{r - 2}^2  - 4\mathcal{T}_{r - 3}^2  + \mathcal{T}_{r - 5}^2  + \mathcal{T}_{r - 6}^2  = 0\,,
\end{equation}
\begin{equation}\label{eq.gd5w62t}
\mathcal{T}_{r + 1}^2  - 4\mathcal{T}_r^2  + 2\mathcal{T}_{r - 2}^2  + 16\mathcal{T}_{r - 3}^2  + 4\mathcal{T}_{r - 4}^2  - 2\mathcal{T}_{r - 6}^2  - \mathcal{T}_{r - 7}^2  = 0
\end{equation}
and
\begin{equation}\label{eq.k3ltayi}
\mathcal{T}_{r + 2}^2  - 2\mathcal{T}_{r + 1}^2  - 2\mathcal{T}_r^2  - 8\mathcal{T}_{r - 1}^2  - 2\mathcal{T}_{r - 2}^2  - 6\mathcal{T}_{r - 3}^2  + 2\mathcal{T}_{r - 4}^2  + \mathcal{T}_{r - 6}^2  = 0\,.
\end{equation}

\end{thm}
\begin{proof}
To prove identity~\eqref{eq.jesyg5e}, write
\begin{equation}\label{eq.u5dlvn9}
\mathcal{T}_r-\mathcal{T}_{r-1}=\mathcal{T}_{r-2}+\mathcal{T}_{r-3}\,,
\end{equation}
square both sides and use the identity~\eqref{eq.pbljs54} to resolve the cross-products $\mathcal{T}_r \mathcal{T}_{r - 1} $ and $\mathcal{T}_{r-2} \mathcal{T}_{r - 3} $. 

\bigskip

Identity~\eqref{eq.ltwcc00} is proved by multiplying through identity~\eqref{eq.o6ns6qw} by $8\mathcal{T}_{r-2}$, obtaining $8\mathcal{T}_r \mathcal{T}_{r - 2}  = 16\mathcal{T}_{r - 1} \mathcal{T}_{r - 2}  - 8\mathcal{T}_{r - 4} \mathcal{T}_{r - 2} $ and using identities~\eqref{eq.pbljs54} and~\eqref{eq.y9lf0gg} to resolve the cross-products.

\bigskip

To prove identity~\eqref{eq.pq5tubq}, rearrange relation~\eqref{eq.u5dlvn9} as $\mathcal{T}_r-\mathcal{T}_{r-2}=\mathcal{T}_{r-1}+\mathcal{T}_{r-3}$, square both sides and use the identity~\eqref{eq.y9lf0gg} to resolve the cross-products. 

\bigskip

Identity~\eqref{eq.gd5w62t} is proved by rearranging relation~\eqref{eq.u5dlvn9} as $\mathcal{T}_r-\mathcal{T}_{r-3}=\mathcal{T}_{r-1}+\mathcal{T}_{r-2}$, multiplying through by $\mathcal{T}_{r-3}$ to obtain $\mathcal{T}_r\mathcal{T}_{r-3}-\mathcal{T}_{r-3}^2=\mathcal{T}_{r-1}\mathcal{T}_{r-3}+\mathcal{T}_{r-2}\mathcal{T}_{r-3}$ and using identities~\eqref{eq.pbljs54},~\eqref{eq.a41lrc2} and~\eqref{eq.y9lf0gg} to resolve the cross-products. 

\bigskip

To prove identity~\eqref{eq.k3ltayi} multiply through the recurrence relation~\eqref{eq.ffnh4ol} by $\mathcal{T}_r$ to obtain $\mathcal{T}_r^2-\mathcal{T}_r\mathcal{T}_{r-1}-\mathcal{T}_r\mathcal{T}_{r-2}-\mathcal{T}_r\mathcal{T}_{r-3}=0$. Now use the identities~\eqref{eq.pbljs54}, \eqref{eq.a41lrc2} and~\eqref{eq.y9lf0gg} to resolve the cross-products.

\end{proof}
\section{Partial sums}
From the identity~\eqref{eq.jesyg5e} of Theorem~\ref{thm.eo07gjo} and using $\{X_r\}\equiv\{\mathcal{T}^2_r\}$, $r\in\Z$, in Lemma~\ref{lemma.qm8k37h}, the next result follows.
\begin{thm}\label{thm.awteiu4}
The following identity holds for integer $k$ and arbitrary $x$:
\[
\begin{split}
&(1 - 3x - x^2  - x^3 )(1 + x + x^2  - x^3 )\sum_{j = 0}^k {x^j \mathcal{T}_j^2 }\\
&\qquad= (2 + 3x + 6x^2  - x^3  - x^5 )(\mathcal{T}_{ - 1}{}^2  - x^{k + 1} \mathcal{T}_k^2 )
\end{split}
\]
\[
\begin{split}
&\qquad\quad + (3 + 6x - x^2  - x^4 )(\mathcal{T}_{ - 2}{}^2  - x^{k + 1} \mathcal{T}_{k - 1}{}^2 )\\
&\qquad\quad + (6 - x - x^3 )(\mathcal{T}_{ - 3}{}^2  - x^{k + 1} \mathcal{T}_{k - 2}{}^2 )
\end{split}
\]
\[
\begin{split}
&\qquad\qquad\qquad- (1 + x^2 )(\mathcal{T}_{ - 4}{}^2  - x^{k + 1} \mathcal{T}_{k - 3}{}^2 )\\
&\qquad\qquad\qquad- x(\mathcal{T}_{ - 5}{}^2  - x^{k + 1} \mathcal{T}_{k - 4}{}^2 ) - (\mathcal{T}_{ - 6}{}^2  - x^{k + 1} \mathcal{T}_{k - 5}{}^2 )\,.
\end{split}
\]

\end{thm}
In particular,
\begin{equation}\label{eq.jct38sz}
\begin{split}
8\sum_{j = 0}^k { \mathcal{T}_j^2 } &= 9(\mathcal{T}_k^2  - \mathcal{T}_{ - 1}{}^2 ) + 7(\mathcal{T}_{k - 1}{}^2  - \mathcal{T}_{ - 2}{}^2 ) + 4(\mathcal{T}_{k - 2}{}^2  - \mathcal{T}_{ - 3}{}^2 )\\
&\qquad- 2(\mathcal{T}_{k - 3}{}^2  - \mathcal{T}_{ - 4}{}^2 ) - (\mathcal{T}_{k - 4}{}^2  - \mathcal{T}_{ - 5}{}^2 ) - (\mathcal{T}_{k - 5}{}^2  - \mathcal{T}_{ - 6}{}^2 )
\end{split}
\end{equation}
and
\begin{equation}
\begin{split}
8\sum_{j = 0}^k {(-1)^j \mathcal{T}_j^2 }&= 7(\mathcal{T}_{ - 1}{}^2  + (-1)^k \mathcal{T}_k^2 ) -5(\mathcal{T}_{ - 2}{}^2  + (-1)^k \mathcal{T}_{k - 1}{}^2 )\\
&\qquad + 8(\mathcal{T}_{ - 3}{}^2  + (-1)^k \mathcal{T}_{k - 2}{}^2 )- 2(\mathcal{T}_{ - 4}{}^2  + (-1)^k \mathcal{T}_{k - 3}{}^2 )\\
&\qquad\quad+(\mathcal{T}_{ - 5}{}^2 + (-1)^k \mathcal{T}_{k - 4}{}^2 ) - (\mathcal{T}_{ - 6}{}^2  + (-1)^k \mathcal{T}_{k - 5}{}^2 )\,.
\end{split}
\end{equation}
A variant of~\eqref{eq.jct38sz} was given by Maiorano~\cite{maiorano96}.

\bigskip

For the Tribonacci numbers, these results translate to:
\[
\begin{split}
&(1 - 3x - x^2  - x^3 )(1 + x + x^2  - x^3 )\sum_{j = 0}^k {x^j T_j^2 }\\
&\qquad= -(2 + 3x + 6x^2  - x^3  - x^5 )x^{k + 1} T_k^2 
\end{split}
\]
\begin{equation}\label{eq.im7vx8w}
\begin{split}
&\qquad\quad + (3 + 6x - x^2  - x^4 )(1 - x^{k + 1} T_{k - 1}{}^2 )\\
&\qquad\quad + (6 - x - x^3 )(1 - x^{k + 1} T_{k - 2}{}^2 )
\end{split}
\end{equation}
\[
\begin{split}
&\qquad\qquad\qquad+ (1 + x^2 ) x^{k + 1} T_{k - 3}{}^2 - x(4 - x^{k + 1} T_{k - 4}{}^2 )\\
&\qquad\qquad\qquad - (9 - x^{k + 1} T_{k - 5}{}^2 )\,,
\end{split}
\]
\begin{equation}\label{eq.i6m3xea}
\begin{split}
8\sum_{j = 0}^k {T_j^2 } &= 9T_k^2 + 7T_{k - 1}{}^2 + 4T_{k - 2}{}^2 - 2T_{k - 3}{}^2 - T_{k - 4}{}^2 - T_{k - 5}{}^2 + 2
\end{split}
\end{equation}
and
\begin{equation}\label{eq.g5a20e1}
8\sum_{j = 0}^k {( - 1)^j T_j^2 }  = ( - 1)^k \left(7T_k^2  - 5T_{k - 1}{}^2  + 8T_{k - 2}{}^2  - 2T_{k - 3}{}^2  + T_{k - 4}{}^2  - T_{k - 5}{}^2 \right) - 2\,.
\end{equation}
Shah~\cite{shah11} proved a shorter version of~\eqref{eq.i6m3xea} and reported a variant of same, attributed to Zeitlin~\cite{zeitlin67}.

\bigskip

Combining identity~\eqref{eq.i6m3xea} and identity~\eqref{eq.g5a20e1}, and using the summation identity
\begin{equation}\label{eq.sl69rgj}
\sum_{j = 0}^{2k} {f_j }  = \sum_{j = 0}^k {f_{2j} }  + \sum_{j = 1}^k {f_{2j - 1} }\,,
\end{equation}
we obtain
\begin{equation}\label{eq.ry26tbj}
8\sum_{j = 0}^k {T_{2j}^2 }  = 8T_{2k}^2  + T_{2k - 1}^2  + 6T_{2k - 2}^2  - 2T_{2k - 3}^2  - T_{2k - 5}^2
\end{equation}
and
\begin{equation}\label{eq.btzz0b5}
8\sum_{j = 1}^k {T_{2j - 1}^2 }  = T_{2k}^2  + 6T_{2k - 1}^2  - 2T_{2k - 2}^2  - T_{2k - 4}^2  + 2\,.
\end{equation}
Using $f_j=i^jT_j^2$ in identity~\eqref{eq.sl69rgj}, where $i=\sqrt {-1}$ is the imaginary unit, allows us to write
\begin{equation}\label{eq.o6nyqnz}
\sum_{j = 0}^k {( - 1)^j T_{2j}^2 }  = {\mathop{\rm Re}\nolimits} \sum_{j = 0}^{2k} {i^j T_{j}^2 },\quad \sum_{j = 1}^k {( - 1)^{j - 1} T_{2j-1}^2 }  = {\mathop{\rm Im}\nolimits} \sum_{j = 0}^{2k} {i^j T_{j}^2 }\,.
\end{equation}
Using $x=i$ in identity~\eqref{eq.im7vx8w} and taking note of identity~\eqref{eq.o6nyqnz}, we obtain the following alternating versions of identities~\eqref{eq.ry26tbj} and~\eqref{eq.btzz0b5}:
\begin{equation}\label{eq.ka363zb}
8\sum_{j = 0}^k {( - 1)^j T_{2j}^2 }  = ( - 1)^k \left( {7T_{2k}^2  + 3T_{2k - 1}^2  - 6T_{2k - 2}^2 - T_{2k - 4}^2  + T_{2k - 5}^2 } \right) + 2
\end{equation}
and
\begin{equation}\label{eq.vr0alb9}
8\sum_{j = 1}^k {( - 1)^{j - 1} T_{2j - 1}^2 }  = ( - 1)^k \left( {T_{2k}^2  - 9T_{2k - 1}^2  - 6T_{2k - 2}^2  + T_{2k - 4}^2  + T_{2k - 5}^2 } \right) + 2\,.
\end{equation}
Replacing $k$ with $2k$ in identities~\eqref{eq.ry26tbj} and~\eqref{eq.ka363zb} we find
\begin{equation}
16\sum_{j = 0}^k {T_{4j}^2 }  = 15T_{4k}^2  + 4T_{4k - 1}^2  - 2T_{4k - 3}^2  - T_{4k - 4}^2  + 2\,,
\end{equation}
by their addition, and
\begin{equation}
16\sum_{j = 1}^k {T_{4j - 2}^2 }  = T_{4k}^2  - 2T_{4k - 1}^2  + 12T_{4k - 2}^2  - 2T_{4k - 3}^2  + T_{4k - 4}^2  - 2T_{4k - 5}^2  - 2\,,
\end{equation}
by their subtraction.

\bigskip

Doing similar analysis with identities~\eqref{eq.btzz0b5} and~\eqref{eq.vr0alb9}, we find
\begin{equation}
16\sum_{j = 1}^k {T_{4j - 3}^2 }  = 2T_{4k}^2  - 3T_{4k - 1}^2  - 8T_{4k - 2}^2  + T_{4k - 5}^2  + 4
\end{equation}
and
\begin{equation}
16\sum_{j = 0}^k {T_{4j - 1}^2 }  = 15T_{4k - 1}^2  + 4T_{4k - 2}^2  - 2T_{4k - 4}^2  - T_{4k - 5}^2\,.
\end{equation}
Weighted sums of the form $\sum_{j = 0}^k {j^p\mathcal{T}_j^2 }$, $p$ a non-negative integer, can be evaluated by making $\sum_{j = 0}^k {x^j \mathcal{T}_j^2 }$ subject in the identity of Theorem~\ref{thm.awteiu4}, setting $x=e^y$, differentiating $p$ times with respect to $y$ and then setting $y=0$. The simplest of such evaluations are the following:
\[
8\sum_{j = 0}^k {j\mathcal{T}_j^2 }  = (9k - 2)\mathcal{T}_k^2  + 7(k - 1)\mathcal{T}_{k - 1}{}^2  + 4(k - 2)\mathcal{T}_{k - 2}{}^2
\]
\begin{equation}
\begin{split}
&\qquad - 2k\mathcal{T}_{k - 3}{}^2  - k\mathcal{T}_{k - 4}{}^2  - (k - 1)\mathcal{T}_{k - 5}{}^2\\
&\qquad\quad + 11\mathcal{T}_{ - 1}{}^2  + 14\mathcal{T}_{ - 2}{}^2  + 12\mathcal{T}_{ - 3}{}^2
\end{split}
\end{equation}
\[
\quad- 2\mathcal{T}_{ - 4}{}^2  - \mathcal{T}_{ - 5}{}^2  - 2\mathcal{T}_{ - 6}{}^2
\]
and
\[
\begin{split}
8\sum\limits_{j = 0}^k {j^2 \mathcal{T}_j^2 }  &= (9k^2  - 4k + 6)\mathcal{T}_k^2  + (7k{}^2  - 14k + 7)\mathcal{T}_{k - 1}{}^2\\
&\qquad+ (4k{}^2  - 16k + 10)\mathcal{T}_{k - 2}{}^2  - (2k{}^2  + 6)\mathcal{T}_{k - 3}{}^2
\end{split}
\]
\[
\qquad\quad\qquad- (k{}^2  + 2)\mathcal{T}_{k - 4}{}^2  - (k{}^2  - 2k + 3)\mathcal{T}_{k - 5}{}^2
\]
\begin{equation}
\begin{split}
&\qquad - 19\mathcal{T}_{ - 1}{}^2  - 28\mathcal{T}_{ - 2}{}^2  - 30\mathcal{T}_{ - 3}{}^2\\
&\quad\qquad + 8\mathcal{T}_{ - 4}{}^2  + 3\mathcal{T}_{ - 5}{}^2  + 6\mathcal{T}_{ - 6}{}^2\,.
\end{split}
\end{equation}
In particular,
\begin{equation}
\begin{split}
8\sum_{j = 0}^k {jT_j^2 }  &= (9k - 2)T_k^2  + 7(k - 1)T_{k - 1}{}^2  + 4(k - 2)T_{k - 2}{}^2\\
&\qquad - 2kT_{k - 3}{}^2  - kT_{k - 4}{}^2  - (k - 1)T_{k - 5}{}^2 + 4\\
\end{split}
\end{equation}
and
\[
\begin{split}
8\sum\limits_{j = 0}^k {j^2 T_j^2 }  &= (9k^2  - 4k + 6)T_k^2  + (7k^2  - 14k + 7)T_{k - 1}^2\\
&\qquad+ (4k^2  - 16k + 10)T_{k - 2}^2  - (2k^2  + 6)T_{k - 3}^2
\end{split}
\]
\[
\qquad\quad\qquad\qquad- (k^2  + 2)T_{k - 4}^2  - (k^2  - 2k + 3)T_{k - 5}^2 + 8\,.
\]
\section{Generating function}
\begin{thm}
Let $\mathcal{G}(x)=\sum_{j = 0}^\infty {x^j \mathcal{T}_j^2 }$ be the generating function of the squares of generalized Tribonacci numbers. $\mathcal{G}(x)$ is given by
\[
\begin{split}
&(1 - 3x - x^2  - x^3 )(1 + x + x^2  - x^3 )\mathcal{G}(x)\\
&= (2 + 3x + 6x^2  - x^3  - x^5 )\mathcal{T}_{ - 1}{}^2
\end{split}
\]
\[
\begin{split}
&\qquad\quad + (3 + 6x - x^2  - x^4 )\mathcal{T}_{ - 2}{}^2 + (6 - x - x^3 )\mathcal{T}_{ - 3}{}^2
\end{split}
\]
\[
\begin{split}
\qquad- (1 + x^2 )\mathcal{T}_{ - 4}{}^2 - x\mathcal{T}_{ - 5}{}^2 - \mathcal{T}_{ - 6}{}^2\,,
\end{split}
\]
provided $x^k \mathcal{T}_k^2 $ vanishes as $k$ approaches infinity.

\end{thm}

In particular, we have that $G(x)$, the generating function of the squares of Tribonacci numbers, is given by
\begin{equation}
G(x)=\sum_{j = 0}^\infty  {x^j T_j^2 }  = \frac{{x(1-x-x^2-x^3)}}{{(1 - 3x - x^2  - x^3 )(1 + x + x^2  - x^3 )}}\,,
\end{equation}
valid for $x$ for which $x^kT_k^2 $ vanishes as $k$ approaches infinity.
\section{Cubes of Tribonacci numbers}
We end this note by giving an identity for the cubes of generalized Tribonacci numbers. The identity can be proved by expressing each $\mathcal{T}_{r-j}$, $j\in \{1,2,3,4,5,6,7,9\}$ as a linear combination of $\mathcal{T}_{r}$, $\mathcal{T}_{r-8}$ and $\mathcal{T}_{r-10}$ and substituting into the identity of Theorem~\ref{thm.zcb38b5}.
\begin{thm}\label{thm.zcb38b5}
The following identity holds for integer $r$:
\[
\begin{split}
&\mathcal{T}_r^3  - 4\mathcal{T}_{r - 1}^3  - 9\mathcal{T}_{r - 2}^3  - 34\mathcal{T}_{r - 3}^3  + 24\mathcal{T}_{r - 4}^3  - 2\mathcal{T}_{r - 5}^3\\
&\quad + 40\mathcal{T}_{r - 6}^3  - 14\mathcal{T}_{r - 7}^3  - \mathcal{T}_{r - 8}^3  - 2\mathcal{T}_{r - 9}^3  + \mathcal{T}_{r - 10}^3  = 0\,.
\end{split}
\]
\end{thm}

\end{document}